\documentclass[11pt]{amsart}
\usepackage[T1]{fontenc}
\usepackage[utf8]{inputenc}
\usepackage{lmodern}
\usepackage[margin=1in]{geometry}
\usepackage{amsmath,amssymb,amsthm}
\usepackage{microtype}
\usepackage[hidelinks]{hyperref}

\newtheorem{theorem}{Theorem}[section]
\newtheorem{corollary}[theorem]{Corollary}
\newtheorem{lemma}[theorem]{Lemma}
\newtheorem{proposition}[theorem]{Proposition}
\newcommand{\R}{\mathbb R}

\newcommand{\T}{\mathsf T}
\DeclareMathOperator{\diag}{diag}
\DeclareMathOperator{\Ree}{Re}

\title[Counterexamples to the Ramos conjecture for two hyperplanes]{Counterexamples to the Ramos conjecture\\for two hyperplanes}
\date{September 22, 2026}
\author{Florian Frick}
\address{Dept. Math. Sciences, Carnegie Mellon University, Pittsburgh, PA 15213, USA}
\email{frick@cmu.edu}
\hypersetup{pdftitle={Counterexamples to the Ramos conjecture for two hyperplanes}}

\begin{document}

\begin{abstract}
For every $n\ge2$, we construct $4n-2$ nondegenerate Gaussian measures
on $\R^{6n-3}$ that cannot be simultaneously equipartitioned by two
affine hyperplanes. This disproves the Ramos conjecture for two
hyperplanes. Combined with known upper bounds, the construction shows
that $3\cdot2^{s-1}-2$ is the least dimension guaranteeing a common
two-hyperplane equipartition of $2^s-2$ absolutely continuous probability
measures, for every $s\ge3$. We characterize the Gaussian equipartition threshold
in terms of the least number of positive definite quadratic measurements
needed for phase retrieval. Modified complex polynomial multiplication
gives $2r-2$ positive definite measurements
in every even dimension $r\ge4$. This number is optimal when
$r=2^k+2$, $k\ge1$.
\end{abstract}

\maketitle

\section{Introduction}

Two affine hyperplanes \emph{equipartition} a probability measure on
$\R^d$ if their complement consists of four open regions,
each of measure~$1/4$.
Let $\Delta(m,2)$ be the least dimension $d$ such that every collection
of $m$ absolutely continuous probability measures on $\R^d$ admits a
common equipartition by two affine hyperplanes. The problem originates
in Gr\"unbaum's questions on hyperplane partitions~\cite{grunbaum}.
Hadwiger~\cite{hadwiger} proved $\Delta(2,2)=3$.
Ramos~\cite{ramos} conjectured that
\[
\Delta(m,2)=\left\lceil\frac{3m}{2}\right\rceil.
\]
The lower bound follows from placing measures along the moment curve $t\mapsto (t, t^2, \dots, t^d)$; see Avis and
Ramos~\cite{avis,ramos}: each of $m$ disjoint intervals carrying a measure
requires at least three cuts, while two hyperplanes in $\R^d$ meet the
curve at most $2d$ times. Approximation gives absolutely continuous
examples. The conjectured equality holds when $m$ differs from a power
of two by at most one; see~\cite{mlvz, bfhz-topology, bfhz-relative}.

A collection of $k$ hyperplanes equipartitions a probability measure on~$\R^d$ if their complement consists of $2^k$ open regions, each of measure~$2^{-k}$. Let $\Delta(m,k)$ denote the least dimension where any $m$ absolutely continuous probability measures admit an equipartition by $k$ hyperplanes. Ramos conjectured more generally that $\Delta(m,k) = \lceil\frac{m(2^k-1)}{k}\rceil$. The lower bound again follows from placing measures along the moment curve.

Sober\'on~\cite{soberon} recently disproved this more general conjecture by constructing a smooth, strictly positive density
on $\R^4$ that cannot be equipartitioned by four affine hyperplanes. Together with the upper bound of~\cite{bfhz-relative} this shows $\Delta(1,4) = 5$. 
His construction perturbs one Gaussian by cubic and quartic terms to
obstruct simultaneous vanishing on orthogonal frames. The construction here uses several unperturbed Gaussians. Their centers constrain the cutting
normals, and their covariance forms give an exact bilinear obstruction.
We show that the conjecture fails for two hyperplanes for every
$m\equiv2\pmod4$ with $m\ge6$.

\begin{theorem}\label{thm:main}
For every integer $n\ge2$, there are $4n-2$ nondegenerate Gaussian
probability measures on $\R^{6n-3}$ with no common equipartition by two
affine hyperplanes. Consequently,
\[
\Delta(4n-2,2)\ge6n-2.
\]
\end{theorem}

The first instance consists of six measures in $\R^9$. All measures in
the theorem have smooth, everywhere positive densities. The construction
also determines an infinite family of exact equipartition dimensions. For general $m$, Mani-Levitska, Vre\'cica, and
\v Zivaljevi\'c~\cite[Theorem~39]{mlvz} proved
$\Delta(m,2)\le m+2^{\lfloor\log_2m\rfloor}$. This agrees with the moment curve lower bound for $m = 2^s-1$, and agrees with the new lower bounds above for $m=2^s-2$ and $s\ge3$. Thus we obtain:

\begin{corollary}\label{cor:exact}
For every integer $s\ge3$,
\[
\Delta(2^s-2,2)=3\cdot2^{s-1}-2.
\]
\end{corollary}

The first new exact value that follows from this is $\Delta(6,2)=10$.
Simon~\cite{simon} proved the same exact values as in Corollary~\ref{cor:exact} when the
two cutting hyperplanes are required to be perpendicular.

The relative obstruction argument in~\cite[Theorem~1.5]{bfhz-relative} establishes sharp
upper bounds for~$\Delta(m,2)$, whenever the number of masses~$m$ deviates from a power of two by at most one. Also see the companion paper~\cite{bfhz-topology}, which surveys the literature at the time. Frick, Murray, Simon, and
Stemmler~\cite{fmss} extend the sharp upper bounds to hyperplane
transversal theorems, while Blagojevi\'c and Crabb~\cite{bc}
study continuously parameterized masses on vector bundles.
For a broader survey of mass partitions, see Rold\'an-Pensado and
Sober\'on~\cite{survey}.

The geometry of the construction is as follows. Any hyperplane that
bisects a nondegenerate Gaussian contains its center. Choosing $m$
affinely independent centers therefore confines the normals of both
cutting hyperplanes to a subspace of codimension $m-1$. On that subspace,
each covariance matrix imposes an orthogonality condition on the two
normals. We choose the covariance matrices so that these conditions have
no common solution with both normals nonzero. A modification of complex
polynomial multiplication gives these forms in every required dimension. A related geometric construction appears in Matschke’s dimension-reduction argument~\cite[Lemma 2.2]{matschke}, where an auxiliary ball forces all cutting hyperplanes through its center.

\section{Phase retrieval}

Phase retrieval asks whether a vector is determined, up to an unavoidable
sign, by quadratic measurements. For real symmetric matrices
$A=(A_1,\ldots,A_m)$ on $\R^r$, write
\[
\Phi_A(x)=(x^{\T}A_1x,\ldots,x^{\T}A_mx).
\]
The family $A$ \emph{does phase retrieval} if
$\Phi_A(x)=\Phi_A(y)$ implies $y=\pm x$ for all $x,y\in\R^r$.

The following criterion connects this problem to Gaussian equipartitions.
For orthogonal projections it is due to Edidin~\cite{edidin};
the general symmetric case is given by Wang and
Xu~\cite{wang-xu}.

\begin{proposition}\label{prop:phase}
A family of real symmetric matrices $A_1,\ldots,A_m$ on $\R^r$ does
phase retrieval if and only if
\begin{equation}\label{eq:nonsingular}
u^{\T}A_iv=0\quad\text{for every }i
\qquad\Longrightarrow\qquad u=0\text{ or }v=0.
\end{equation}
Equivalently, $A_1u,\ldots,A_mu$ span $\R^r$ for every $u\ne0$.
\end{proposition}

\begin{proof}
Polarization gives
\[
\Phi_A(u+v)-\Phi_A(u-v)
=4\bigl(u^{\T}A_1v,\ldots,u^{\T}A_mv\bigr).
\]
Here $u+v=\pm(u-v)$ precisely when $u=0$ or $v=0$.
Conversely, any pair $x,y$ can be written as $u+v,u-v$.
The spanning condition follows by taking orthogonal complements.
\end{proof}

Thus phase retrieval is equivalent to nonsingularity of a symmetric
bilinear map. Write $p_+(r)$ for the minimum number of positive
semidefinite quadratic measurements doing phase retrieval on~$\R^r$.
The same minimum is obtained if all matrices are required to be positive
definite. Indeed, for a positive semidefinite phase-retrieving family
$Q_1,\ldots,Q_m$, the sum $S=\sum_iQ_i$ is positive definite, since a 
common kernel vector would have the same measurements as zero. Set
$A_i=S+Q_i$. These positive definite measurements contain exactly the
same information, because
\begin{equation}\label{eq:positive}
Q_i=A_i-\frac1{m+1}\sum_j A_j.
\end{equation}

\section{The equivalence for Gaussian measures}

Let $\mu_i=N(c_i,\Sigma_i)$, $1\le i\le m$, be nondegenerate Gaussian
probability measures on~$\R^d$. Recall that the covariance of real
random variables $Y$ and $Z$ with finite second moments is
\[
  \operatorname{Cov}(Y,Z)
  = \mathbb{E}\bigl[(Y-\mathbb{E}Y)(Z-\mathbb{E}Z)\bigr].
\]
For a random vector $X_i$ with distribution $\mu_i$, its mean is $c_i$
and its covariance matrix is
\[
  \Sigma_i=\mathbb{E}[(X_i-c_i)(X_i-c_i)^{\T}].
\]
The common normal space determined by the centers is
\[
  V=\operatorname{span}\{c_i-c_1:1\le i\le m\}^{\perp}.
\]
Restricting the covariance form to $u,v\in V$ gives the symmetric
operator $A_i\colon V\to V$ characterized by
\[
  \langle u,A_i v\rangle
  =u^{\T}\Sigma_i v
  =\operatorname{Cov}\bigl(\langle u,X_i\rangle,
                          \langle v,X_i\rangle\bigr),
  \qquad u,v\in V.
\]
Thus $A_i=\pi_V\Sigma_i|_V$, where $\pi_V$ denotes orthogonal
projection onto~$V$; the projection is needed because $\Sigma_i$
need not preserve~$V$. Equivalently, $A_i$ is the covariance
operator of the projected random vector~$\pi_VX_i$.
Each $A_i$ is positive definite on~$V$, since
$\langle u,A_i u\rangle=u^{\T}\Sigma_i u>0$ for every nonzero $u\in V$.

If $V=\{0\}$, no hyperplane contains all centers,
so there is no common bisector. Otherwise, we record the following
immediate consequence of the standard Gaussian orthogonality criterion
and Proposition~\ref{prop:phase}. We include the proof for completeness.

\begin{proposition}\label{prop:gaussian}
Assume $\dim V\ge1$. The measures $\mu_1,\ldots,\mu_m$ admit a common
equipartition by two affine hyperplanes if and only if
$A_1,\ldots,A_m$ fail to do phase retrieval on $V$.
\end{proposition}

\begin{proof}
For the standard Gaussian on $\R^d$, two hyperplanes give an
equipartition exactly when they contain the origin and their normals
are perpendicular. This standard fact is used by
Sober\'on~\cite{soberon}.
Indeed, a bisecting hyperplane contains the origin by uniqueness of the
median of every nonzero one-dimensional Gaussian projection.
Coincident hyperplanes cannot give an equipartition. For distinct
hyperplanes through the origin, rotational symmetry in the plane
spanned by their normals gives four equal sectors exactly when the
normals are perpendicular.

For $N(c,\Sigma)$ with $\Sigma$ positive definite, the characterization
follows by an invertible affine change of variables. Set
$B=\Sigma^{1/2}$. If $Z$ is standard Gaussian,
then $c+BZ$ has covariance $BB^{\T}=\Sigma$ and distribution
$N(c,\Sigma)$. Under $x=c+Bz$, the hyperplanes $u^{\T}x=a$ and
$v^{\T}x=b$, with $u,v\ne0$, pull back to
\[
  (B^{\T}u)^{\T}z=a-u^{\T}c,
  \qquad
  (B^{\T}v)^{\T}z=b-v^{\T}c.
\]
This change of variables preserves region probabilities. The standard
Gaussian criterion therefore gives an equipartition exactly when
$a=u^{\T}c$, $b=v^{\T}c$, and
\begin{equation}\label{eq:gaussian}
  u^{\T}\Sigma v=\langle B^{\T}u,B^{\T}v\rangle=0.
\end{equation}
Thus both hyperplanes contain $c$, and their normals are orthogonal
for the covariance inner product.

Both hyperplanes in a common equipartition must therefore contain every
$c_i$. They have equations $\langle u,x-c_1\rangle=0$ and
$\langle v,x-c_1\rangle=0$ for nonzero $u,v\in V$.
By~\eqref{eq:gaussian}, these hyperplanes equipartition every measure
exactly when $u^{\T}A_iv=0$ for every $i$.
Conversely, any such pair defines a common equipartition. Positive
definiteness ensures that $u$ and $v$ are linearly independent.
Proposition~\ref{prop:phase} now proves the claim.
\end{proof}

In particular, an ambiguous pair of signals $x\ne\pm y$ for the
restricted covariance forms gives cutting normals $x+y$ and $x-y$.
Conversely, cutting normals $u,v$ give indistinguishable signals $u+v$
and $u-v$. Thus \emph{successful} phase retrieval corresponds to
\emph{nonexistence} of a Gaussian equipartition. The equivalence holds
for each fixed Gaussian family.

We next pass from a phase-retrieving family to measures in the desired
dimension. 

\begin{lemma}\label{lem:transfer}
If $m$ positive semidefinite symmetric forms on $\R^r$ do phase
retrieval, then there are $m$ nondegenerate Gaussian measures on
$\R^{m+r-1}$ with no common two-hyperplane equipartition.
Consequently, $\Delta(m,2)\ge m+r$.
\end{lemma}

\begin{proof}
By~\eqref{eq:positive}, we may assume that the given matrices
$A_1,\ldots,A_m$ are positive definite.
Write $\R^{m+r-1}=\R^{m-1}\oplus\R^r$. Let
$a_1,\ldots,a_{m-1}$ be the standard basis of $\R^{m-1}$ and $a_m=0$.
Set
\[
\mu_i=N\!\left((a_i,0),\diag(I_{m-1},A_i)\right),
\qquad 1\le i\le m.
\]
For these centers, $V=0\oplus\R^r$, and the restricted covariance
forms are $A_i$. Proposition~\ref{prop:gaussian} excludes a common
equipartition.
\end{proof}

Xu~\cite[Theorem~4.2]{xu} constructed six rank-two orthogonal
projections on $\R^4$ that do phase retrieval. Six projections onto
hyperplanes with the same property were constructed in
\cite[Theorem~4.2]{hyperplanes}. Since orthogonal projections are
positive semidefinite, Lemma~\ref{lem:transfer} turns either family
into six nondegenerate Gaussian measures in $\R^9$ with no common
two-hyperplane equipartition. For the full family of counterexamples,
we construct positive definite forms directly, without requiring them
to be projections.

\subsection{Characterizing the Gaussian threshold}

Let $\Delta_G(m,2)$ denote the least dimension guaranteeing an
equipartition by two affine hyperplanes for every $m$ nondegenerate Gaussians, with arbitrary
centers. Recall that $p_+(r)$ denotes the minimum number of positive
semidefinite quadratic measurements doing phase retrieval on~$\R^r$. Set
\[
\rho(m)=\max\{r\ge1:p_+(r)\le m\}.
\]
This maximum exists: $p_+(1)=1$, and the spanning criterion gives
$p_+(r)\ge r$. Restriction to a subspace preserves phase retrieval,
so $p_+$ is nondecreasing.

\begin{theorem}\label{thm:threshold}
For every $m\ge1$, we have
\[
\Delta_G(m,2)=m+\rho(m).
\]
Consequently,
\begin{align}
p_+(r)\le m&\quad\Longrightarrow\quad\Delta(m,2)\ge m+r,
\label{eq:forward}\\
\Delta(m,2)\le m+r-1&\quad\Longrightarrow\quad p_+(r)\ge m+1.
\label{eq:reverse}
\end{align}
\end{theorem}

\begin{proof}
Lemma~\ref{lem:transfer} gives a Gaussian counterexample in dimension
$m+\rho(m)-1$; repeat measurements if fewer than $m$ are needed.
Conversely, in dimension $m+\rho(m)$ the common normal space of any
$m$ centers has dimension at least $\rho(m)+1$. The compressed
covariance forms cannot do phase retrieval, by the definition of
$\rho(m)$, so Proposition~\ref{prop:gaussian} gives an equipartition.
The Gaussian equipartition property is monotone in dimension: project
onto a lower-dimensional space and lift the cutting hyperplanes.
This proves the identity. Since $\Delta_G(m,2)\le\Delta(m,2)$,
the first implication follows; the second is its contrapositive.
\end{proof}

Thus a small phase-retrieving family produces an equipartition
counterexample, while a lower bound on measurement numbers forces
Gaussian equipartitions.

\section{Positive definite forms from polynomial multiplication}

Wang and Xu~\cite[Theorem~5.1(i)]{wang-xu} record that $2r-2$
symmetric quadratic measurements suffice for phase retrieval in every
even dimension $r$. Their statement allows indefinite forms.
To apply Lemma~\ref{lem:transfer}, we need a nonsingular system whose
span contains a positive definite form. We obtain such a system by
changing one coordinate of polynomial multiplication.

\begin{lemma}\label{lem:polynomials}
For every $n\ge2$, there are $4n-2$ positive definite symmetric
real bilinear forms on~$\R^{2n}$ satisfying~\eqref{eq:nonsingular}, equivalently, their associated quadratic measurements do phase retrieval.
\end{lemma}

\begin{proof}
Identify $\R^{2n}$ with the space of complex polynomials of degree
less than $n$. For
$p(t)=\sum_{a=0}^{n-1}p_at^a$ and
$q(t)=\sum_{a=0}^{n-1}q_at^a$, write
\[
c_j(p,q)=\sum_{a+b=j}p_aq_b,
\qquad 0\le j\le2n-2,
\]
for the coefficients of $pq$. Their real and imaginary parts are
$4n-2$ symmetric real bilinear forms. 
Their span contains no positive definite form: every real linear
combination $Q$ of these forms satisfies $Q(ip,ip)=-Q(p,p)$.
Retain all these $4n-2$ symmetric real bilinear forms except
$\Ree c_1$, and denote these $4n-3$ forms by $T_i$. Replace the
discarded form by
\[
D(p,q)=\Ree\sum_{a=0}^{n-1}p_a\overline{q_a}
+\frac12\Ree c_1(p,q).
\]
This form is positive definite, since
\[
D(p,p)=\sum_a|p_a|^2+\Ree(p_0p_1)
\ge\frac12\sum_a|p_a|^2.
\]

Suppose that $D(p,q)=T_i(p,q)=0$ for every $i$, with $p,q\ne0$.
The retained coefficient equations give $pq=\lambda t$ for a real
$\lambda\ne0$. Thus up to interchanging $p$ and $q$, we have $p=a$ and $q=bt$ with
$ab=\lambda$. Their coefficient supports are disjoint. Thus the
first term of $D(p,q)$ vanishes, while its second term is
$\lambda/2\ne0$, a contradiction.

Finally, the coefficient bound
\[
|c_j(p,p)|\le\sum_{a+b=j}|p_a||p_b|\le\sum_a|p_a|^2
\]
shows that the $4n-2$ forms
\[
A_1=D,\qquad A_{i+1}=D+\frac14T_i
\quad(1\le i\le4n-3)
\]
satisfy $A_i(p,p)\ge\frac14\sum_a|p_a|^2$ and are therefore positive
definite. If $A_i(p,q)=0$ for every $i$, then $D(p,q)=0$
and $T_i(p,q)=0$ for every $i$, so the preceding argument
proves~\eqref{eq:nonsingular}.
\end{proof}

\begin{proof}[Proof of Theorem~\ref{thm:main}]
Apply Lemma~\ref{lem:transfer} to the forms in
Lemma~\ref{lem:polynomials}, with $m=4n-2$ and $r=2n$.
The resulting measures lie in dimension $m+r-1=6n-3$.
\end{proof}

The counterexamples may also be chosen to have compact support.
Let $B^d$ be the Euclidean unit ball and replace each Gaussian
$N(c_i,\Sigma_i)$ by the uniform probability measure on
\[
E_i=c_i+\Sigma_i^{1/2}B^d.
\]
This measure has center $c_i$ and covariance matrix $\Sigma_i/(d+2)$.
Every bisecting hyperplane contains~$c_i$, and an affine transformation
to $B^d$ shows that two such hyperplanes equipartition the measure
exactly when their normals satisfy $u^{\T}\Sigma_i v=0$.
Thus the same bilinear obstruction applies.

\section{Consequences for phase retrieval}

Lemma~\ref{lem:polynomials} gives $p_+(r)\le 2r-2$ for every even
$r\ge4$. The correspondence with Gaussian equipartitions shows that
this construction is optimal in infinitely many dimensions.

\begin{corollary}\label{cor:phase-counts}
For every integer $k\ge1$,
\[
p_+(2^k+1)=2^{k+1}+1,
\qquad
p_+(2^k+2)=2^{k+1}+2.
\]
Both minima can be attained with all measurements positive definite.
\end{corollary}

\begin{proof}
The sharp equipartition values~\cite{bfhz-relative}
\[
\Delta(2^{k+1},2)=3\cdot2^k,
\qquad
\Delta(2^{k+1}+1,2)=3\cdot2^k+2
\]
give the two lower bounds by~\eqref{eq:reverse}, applied respectively
with $(m,r)=(2^{k+1},2^k+1)$ and
$(m,r)=(2^{k+1}+1,2^k+2)$.
Generic families of $2r-1$ rank-one positive semidefinite measurements
do phase retrieval on~$\R^r$; see~\cite{bce}, giving the first
upper bound. Lemma~\ref{lem:polynomials} supplies the second.
Equation~\eqref{eq:positive} allows the rank-one measurements to be
replaced by positive definite forms without increasing their number.
\end{proof}

Wang and Xu~\cite[Theorem~5.1]{wang-xu} established the corresponding
sharp counts for arbitrary real symmetric measurements, using
projective-space nonembedding results for the lower bounds.
The argument above proves the lower bounds in the positive
semidefinite class directly from equipartition theorems.
Our polynomial construction attains the even-dimensional upper
bound with all measurements positive definite. In particular,
$p_+(4)=6$.

\medskip
\noindent\textbf{AI use statement.}
LLMs were used extensively for proof ideation,
literature review, and for drafting and revising the manuscript.


\begin{thebibliography}{99}

\bibitem{avis}
D.~Avis, \emph{Non-partitionable point sets}, Inform. Process. Lett.
\textbf{19} (1984), no.~3, 125--129.
\href{https://doi.org/10.1016/0020-0190(84)90090-5}{doi:10.1016/0020-0190(84)90090-5}.

\bibitem{bce}
R.~Balan, P.~G. Casazza, and D.~Edidin, \emph{On signal reconstruction
without phase}, Appl. Comput. Harmon. Anal. \textbf{20} (2006), no.~3,
345--356.
\href{https://doi.org/10.1016/j.acha.2005.07.001}{doi:10.1016/j.acha.2005.07.001}.

\bibitem{bc}
P.~V.~M. Blagojevi\'c and M.~C. Crabb, \emph{Many partitions of mass
assignments}, Doc. Math. \textbf{30} (2025), no.~1, 41--104.
\href{https://doi.org/10.4171/DM/980}{doi:10.4171/DM/980}.

\bibitem{bfhz-relative}
P.~V.~M. Blagojevi\'c, F.~Frick, A.~Haase, and G.~M. Ziegler,
\emph{Hyperplane mass partitions via relative equivariant obstruction
theory}, Doc. Math. \textbf{21} (2016), 735--771.
\href{https://arxiv.org/abs/1509.02959}{arXiv:1509.02959}.

\bibitem{bfhz-topology}
P.~V.~M. Blagojevi\'c, F.~Frick, A.~Haase, and G.~M. Ziegler,
\emph{Topology of the Gr\"unbaum--Hadwiger--Ramos hyperplane mass
partition problem}, Trans. Amer. Math. Soc. \textbf{370} (2018),
no.~10, 6795--6824.
\href{https://arxiv.org/abs/1502.02975}{arXiv:1502.02975}.

\bibitem{hyperplanes}
S.~Botelho-Andrade, P.~G. Casazza, D.~Cheng, J.~Haas, T.~T. Tran,
J.~C. Tremain, and Z.~Xu, \emph{Phase retrieval by hyperplanes},
in \emph{Frames and Harmonic Analysis},
Contemp. Math. \textbf{706}, Amer. Math. Soc., Providence, RI,
2018, 21--31.
\href{https://doi.org/10.1090/conm/706/14217}
{doi:10.1090/conm/706/14217}.


\bibitem{edidin}
D.~Edidin, \emph{Projections and phase retrieval}, Appl. Comput.
Harmon. Anal. \textbf{42} (2017), no.~2, 350--359.
\href{https://arxiv.org/abs/1506.00674}{arXiv:1506.00674}.

\bibitem{fmss}
F.~Frick, S.~Murray, S.~Simon, and L.~Stemmler,
\emph{Transversal generalizations of hyperplane equipartitions},
Int. Math. Res. Not. IMRN \textbf{2024} (2024), no.~7, 5586--5618.
\href{https://doi.org/10.1093/imrn/rnad216}{doi:10.1093/imrn/rnad216}.

\bibitem{grunbaum}
B.~Gr\"unbaum, \emph{Partitions of mass-distributions and of convex
bodies by hyperplanes}, Pacific J. Math. \textbf{10} (1960), 1257--1261.
\href{https://doi.org/10.2140/pjm.1960.10.1257}{doi:10.2140/pjm.1960.10.1257}.

\bibitem{hadwiger}
H.~Hadwiger, \emph{Simultane Vierteilung zweier K\"orper},
Arch. Math. (Basel) \textbf{17} (1966), 274--278.

\bibitem{mlvz}
P.~Mani-Levitska, S.~Vre\'cica, and R.~T. \v Zivaljevi\'c,
\emph{Topology and combinatorics of partitions of masses by hyperplanes},
Adv. Math. \textbf{207} (2006), no.~1, 266--296.
\href{https://doi.org/10.1016/j.aim.2005.11.013}{doi:10.1016/j.aim.2005.11.013}.

\bibitem{matschke}
B.~Matschke, \emph{A note on masspartitions by hyperplanes},
arXiv preprint (2010).
\href{https://arxiv.org/abs/1001.0193}{arXiv:1001.0193}.

\bibitem{ramos}
E.~A. Ramos, \emph{Equipartition of mass distributions by hyperplanes},
Discrete Comput. Geom. \textbf{15} (1996), no.~2, 147--167.
\href{https://doi.org/10.1007/BF02717729}{doi:10.1007/BF02717729}.

\bibitem{survey}
E.~Rold\'an-Pensado and P.~Sober\'on, \emph{A survey of mass partitions},
Bull. Amer. Math. Soc. (N.S.) \textbf{59} (2022), no.~2, 227--267.
\href{https://arxiv.org/abs/2010.00478}{arXiv:2010.00478}.

\bibitem{simon}
S.~Simon, \emph{Hyperplane equipartitions plus constraints},
J. Combin. Theory Ser. A \textbf{161} (2019), 29--50.
\href{https://doi.org/10.1016/j.jcta.2018.07.012}{doi:\allowbreak10.1016/\allowbreak j.jcta.2018.07.012}.

\bibitem{soberon}
P.~Sober\'on, \emph{Four hyperplanes do not always equipartition a mass
in $\mathbb{R}^4$}, arXiv preprint (2026).
\href{https://arxiv.org/abs/2608.23312}{arXiv:2608.23312}.

\bibitem{wang-xu}
Y.~Wang and Z.~Xu, \emph{Generalized phase retrieval: measurement
number, matrix recovery and beyond}, Appl. Comput. Harmon. Anal.
\textbf{47} (2019), no.~2, 423--446.
\href{https://arxiv.org/abs/1605.08034}{arXiv:1605.08034}.

\bibitem{xu}
Z.~Xu, \emph{The minimal measurement number for low-rank matrix
recovery}, Appl. Comput. Harmon. Anal. \textbf{44} (2018), no.~2,
497--508.
\href{https://arxiv.org/abs/1505.07204}{arXiv:1505.07204}.

\end{thebibliography}
\end{document}